\documentclass[reqno, 12pt]{amsart}
 \usepackage[a4paper, hmargin = 1.25 in, vmargin = 1.5 in]{geometry}
 \usepackage{amsmath, amssymb, amsfonts}
 \usepackage{amsthm}
 
 \newtheorem{theorem}{Theorem}[section]

 \newtheorem{lemma}{Lemma}
 \newtheorem{proposition}{Proposition}
 
 \newtheorem{definition}{Definition}[section]
 \numberwithin{equation}{section} 

\begin{document}

\begin{center}

\begin{title}
\title{\bf\Large{{On a Nabla Fractional Boundary Value Problem with General Boundary Conditions}}}
\end{title}

\vskip 0.25 in

\begin{author}
\author {Jagan Mohan Jonnalagadda\footnote[1]{Department of Mathematics, Birla Institute of Technology and Science Pilani, Hyderabad - 500078, Telangana, India. email: {j.jaganmohan@hotmail.com}}}
\end{author}

\end{center}

\vskip 0.25 in

\noindent{\bf Abstract:} In this article, we consider a nabla fractional boundary value problem with general boundary conditions. Brackins \& Peterson \cite{Br} gave an explicit expression for the corresponding Green's function. Here, we show that this Green's function is nonnegative and obtain an upper bound for its maximum value. Since the expression for the Green's function is complicated, derivation of its properties may not be straightforward. For this purpose, we use a few properties of fractional nabla Taylor monomials. Using the Green's function, we will then develop a Lyapunov-type inequality for the nabla fractional boundary value problem.

\vskip 0.25 in


\vskip 0.25 in


\vskip 0.25 in

\section{Introduction}

Let $a$, $b \in \mathbb{R}$ with $b - a \in \mathbb{N}_{1}$. Consider the homogeneous nabla fractional boundary value problem with general boundary conditions:
\begin{equation}\label{Gen Hom}
\begin{cases}
-\big{(}\nabla^{\nu}_{a}u\big{)}(t) = 0, \quad t \in \mathbb{N}^{b}_{a + 2},\\
\alpha u(a + 1) - \beta (\nabla u)(a + 1) = 0,\\
\gamma u(b) + \delta (\nabla u)(b) = 0,
\end{cases}
\end{equation}
where $1 < \nu < 2$, $\alpha^{2} + \beta^{2} > 0$ and $\gamma^{2} + \delta^{2} > 0$. Brackins \& Peterson \cite{Br} proved that the boundary value problem \eqref{Gen Hom} has only the trivial solution if, and only if 
\begin{equation}\label{Condition}
\xi = (\beta - \alpha)\gamma + \alpha \gamma H_{\nu - 1}(b, a) + \alpha \delta H_{\nu - 2}(b, a) \neq 0.
\end{equation}

In the following theorem, Brackins \& Peterson \cite{Br} gave an explicit expression for its Green's function.

\begin{theorem}[See \cite{Br}]
Assume \eqref{Condition} holds. The Green's function for the boundary value problem \eqref{Gen Hom} is given by
\begin{equation}\label{Gen Green}
G(t, s) = \begin{cases}
u(t, s), \quad t \leq s - 1,\\
v(t, s), \quad t \geq s,
\end{cases}
\end{equation}
where
\begin{multline}\label{Gen Green U}
u(t, s) = \frac{1}{\xi}\Big{[}\alpha \gamma H_{\nu - 1}(t, a)H_{\nu - 1}(b, \rho(s)) + \alpha \delta H_{\nu - 1}(t, a)H_{\nu - 2}(b, \rho(s)) \\ + (\beta - \alpha) \gamma H_{\nu - 1}(b, \rho(s)) + (\beta - \alpha) \delta H_{\nu - 2}(b, \rho(s))\Big{]},
\end{multline}
and
\begin{equation}\label{Gen Green V}
v(t, s) = u(t, s) - H_{\nu - 1}(t, \rho(s)).
\end{equation}
\end{theorem}

We show that this Green's function is nonnegative and obtain an upper bound for its maximum value. Using the Green's function, we will then develop a Lyapunov-type inequality for the nabla fractional boundary value problem
\begin{equation}\label{Gen Non Hom}
\begin{cases}
\big{(}\nabla^{\nu}_{a}u\big{)}(t) + q(t)u(t) = 0, \quad t \in \mathbb{N}^{b}_{a + 2},\\
\alpha u(a + 1) - \beta (\nabla u)(a + 1) = 0,\\
\gamma u(b) + \delta (\nabla u)(b) = 0,
\end{cases}
\end{equation}
where $q : \mathbb{N}^{b}_{a + 1} \rightarrow \mathbb{R}$.

\section{Preliminaries}

We shall use the following notations, definitions and known results of nabla fractional calculus throughout the article. Denote by $\mathbb{N}_{a} : = \{a, a + 1, a + 2, \ldots\}$ and $\mathbb{N}^{b}_{a} : = \{a, a + 1, a + 2, \ldots, b\}$ for any $a$, $b \in \mathbb{R}$ such that $b - a \in \mathbb{N}_{1}$.

\begin{definition}[See \cite{Bo}]
The backward jump operator $\rho : \mathbb{N}_{a} \rightarrow \mathbb{N}_{a}$ is defined by $$\rho(t) = \begin{cases} a, \hspace{0.43 in} t = a, \\ t - 1, \quad t \in \mathbb{N}_{a + 1}.\end{cases}$$
\end{definition}

\begin{definition}[See \cite{Ki, Po}] The Euler gamma function is defined by $$\Gamma (z) : = \int_0^\infty t^{z - 1} e^{-t} dt, \quad \Re(z) > 0.$$ Using its well-known reduction formula, the Euler gamma function can be extended to the half-plane $\Re(z) \leq 0$ except for $z \in \{\ldots, -2, -1, 0\}$. 
\end{definition}

\begin{definition}[See \cite{Go}]
For $t \in \mathbb{R} \setminus \{\ldots, -2, -1, 0\}$ and $r \in \mathbb{R}$ such that $(t + r) \in \mathbb{R} \setminus \{\ldots, -2, -1, 0\}$, the generalized rising function is defined by
\begin{equation}
\nonumber t^{\overline{r}} = \frac{\Gamma(t + r)}{\Gamma(t)}.
\end{equation}
Also, if $t \in \{\ldots, -2, -1, 0\}$ and $r \in \mathbb{R}$ such that $(t + r) \in \mathbb{R} \setminus \{\ldots, -2, -1, 0\}$, then we use the convention that $t^{\overline{r}} : = 0$.
\end{definition}

\begin{definition}[See \cite{Go}]
Let $\mu \in \mathbb{R} \setminus \{\ldots, -2, -1\}$. Define the $\mu^{th}$-order nabla fractional Taylor monomial by $$H_{\mu}(t, a) = \frac{(t - a)^{\overline{\mu}}}{\Gamma(\mu + 1)},$$ provided the right-hand side exists. Observe that $H_{\mu}(a, a) = 0$ and $H_{\mu}(t, a) = 0$ for all $\mu \in \{\ldots, -2, -1\}$ and $t \in \mathbb{N}_{a}$.
\end{definition}

\begin{definition}[See \cite{Bo}]
Let $u: \mathbb{N}_{a} \rightarrow \mathbb{R}$ and $N \in \mathbb{N}_1$. The first order backward (nabla) difference of $u$ is defined by $$\big{(}\nabla u\big{)}(t) : = u(t) - u(t - 1), \quad t \in \mathbb{N}_{a + 1},$$ and the $N^{th}$-order nabla difference of $u$ is defined recursively by $$\big{(}{\nabla}^{N}u\big{)}(t) : = \Big{(}\nabla\big{(}\nabla^{N - 1}u\big{)}\Big{)}(t), \quad t\in \mathbb{N}_{a + N}.$$
\end{definition}

\begin{definition}[See \cite{Go}]
Let $u: \mathbb{N}_{a + 1} \rightarrow \mathbb{R}$ and $N \in \mathbb{N}_1$. The $N^{\text{th}}$-order nabla sum of $u$ based at $a$ is given by
\begin{equation}
\nonumber \big{(}\nabla ^{-N}_{a}u\big{)}(t) : = \sum^{t}_{s = a + 1}H_{N - 1}(t, \rho(s))u(s), \quad t \in \mathbb{N}_{a},
\end{equation}
where by convention $\big{(}\nabla ^{-N}_{a}u\big{)}(a) = 0$. We define $\big{(}\nabla ^{-0}_{a}u\big{)}(t) = u(t)$ for all $t \in \mathbb{N}_{a + 1}$.
\end{definition}

\begin{definition}[See \cite{Go}]
Let $u: \mathbb{N}_{a + 1} \rightarrow \mathbb{R}$ and $\nu > 0$. The $\nu^{\text{th}}$-order nabla sum of $u$ based at $a$ is given by
\begin{equation}
\nonumber \big{(}\nabla ^{-\nu}_{a}u\big{)}(t) : = \sum^{t}_{s = a + 1}H_{\nu - 1}(t, \rho(s))u(s), \quad t \in \mathbb{N}_{a},
\end{equation}
where by convention $\big{(}\nabla ^{-\nu}_{a}u\big{)}(a) = 0$.
\end{definition}

\begin{definition}[See \cite{Go}]
Let $u: \mathbb{N}_{a + 1} \rightarrow \mathbb{R}$, $\nu > 0$ and choose $N \in \mathbb{N}_1$ such that $N - 1 < \nu \leq N$. The $\nu^{\text{th}}$-order nabla difference of $u$ is given by
\begin{equation}
\nonumber \big{(}\nabla ^{\nu}_{a}u\big{)}(t) : = \Big{(}\nabla^N\big{(}\nabla_{a}^{-(N - \nu)}u\big{)}\Big{)}(t), \quad t\in\mathbb{N}_{a + N}.
\end{equation}
\end{definition}

The following properties of gamma function, generalized rising function, and fractional nabla Taylor monomial will be used in Section 3.

\begin{proposition}[See \cite{Go}]\label{Gamma}
Assume the following generalized rising functions and fractional nabla Taylor monomials are well defined. 
\begin{enumerate}
\item $\Gamma(t) > 0$ for $t > 0$, and $\Gamma(t) < 0$ for $-1 < t < 0$.
\item $t^{\overline\nu}(t + \nu)^{\overline \mu} = t^{\overline {\nu + \mu}}$.
\item $\nabla(\nu + t)^{\overline {\mu}} = \mu(\nu + t)^{\overline {\mu - 1}}$.
\item $\nabla(\nu - t)^{\overline {\mu}} = -\mu(\nu - \rho(t))^{\overline {\mu - 1}}$.
\item $\nabla H_{\mu}(t, a) = H_{\mu - 1}(t, a)$.
\item $\nabla H_{\mu}(t, a) - H_{\mu - 1}(t, a) = H_{\mu}(t, a + 1)$.
\item $\sum^{t}_{s = a + 1}H_{\mu}(s, a) = H_{\mu + 1}(t, a)$.
\item $\sum^{t}_{s = a + 1}H_{\mu}(t, \rho(s)) = H_{\mu + 1}(t, a)$.
\end{enumerate}
\end{proposition}

\begin{proposition}[See \cite{Go}]\label{Power Rule}
Let $\nu \in \mathbb{R}^{+}$ and $\mu \in \mathbb{R}$ such that $\mu$, $\mu + \nu$ and $\mu - \nu$ are nonnegative integers. Then, for all $t \in \mathbb{N}_{a}$,
\begin{enumerate}
\item[(i)] $\nabla_{a}^{-\nu}(t - a)^{\overline{\mu}} = \frac{\Gamma(\mu + 1)}{\Gamma(\mu + \nu + 1)}(t - a)^{\overline{\mu + \nu}}$.
\item[(ii)] $\nabla_{a}^{\nu}(t - a)^{\overline{\mu}} = \frac{\Gamma(\mu + 1)}{\Gamma(\mu - \nu + 1)}(t - a)^{\overline{\mu - \nu}}$.
\item[(iii)] $\nabla_{a}^{-\nu} H_{\mu}(t, a) = H_{\mu + \nu}(t, a)$.
\item[(iv)] $\nabla_{a}^{\nu} H_{\mu}(t, a) = H_{\mu - \nu}(t, a)$.
\end{enumerate}
\end{proposition}

\begin{proposition}[See \cite{Ik 2}]\label{H}
Let $\mu > -1$ and $s \in \mathbb{N}_{a}$. Then, the following hold:
\begin{enumerate}
\item[(a)] If $t \in \mathbb{N}_{\rho(s)}$, then $H_{\mu}(t, \rho(s)) \geq 0$, and if $t \in \mathbb{N}_{s}$, then $H_{\mu}(t, \rho(s)) > 0$.
\item[(b)] If $t \in \mathbb{N}_{\rho(s)}$ and $\mu > 0$, then $H_{\mu}(t, \rho(s))$ is a decreasing function of $s$.
\item[(c)] If $t \in \mathbb{N}_{s}$ and $-1 < \mu < 0$, then $H_{\mu}(t, \rho(s))$ is an increasing function of $s$.
\item[(d)] If $t \in \mathbb{N}_{\rho(s)}$ and $\mu \geq 0$, then $H_{\mu}(t, \rho(s))$ is a nondecreasing function of $t$.
\item[(e)] If $t \in \mathbb{N}_{s}$ and $\mu > 0$, then $H_{\mu}(t, \rho(s))$ is an increasing function of $t$.
\item[(f)] If $t \in \mathbb{N}_{s + 1}$ and $-1 < \mu < 0$, then $H_{\mu}(t, \rho(s))$ is a decreasing function of $t$.
\end{enumerate}
\end{proposition}

\begin{proposition}[See \cite{Ik 2}]\label{Mono}
If $0 < \nu \leq \mu$, then $H_{\nu}(t, a) \le H_{\mu}(t, a)$, for each fixed $t \in \mathbb{N}_{a}$.
\end{proposition}

\begin{proposition}[See \cite{Ik 2}]\label{Max}
Let $f$, $g$ be nonnegative real-valued functions on a set $S$. Moreover, assume $f$ and $g$ attain their maximum in $S$. Then, for each fixed $t \in S$, $$\big{|}f(t) - g(t)\big{|} \leq \max \big{\{}f(t), g(t)\big{\}} \leq \max \Big{\{}\max_{t \in S}f(t), \max_{t \in S}g(t)\Big{\}}.$$
\end{proposition}

\begin{proposition}\label{h}
Let $\mu > -1$, $s \in \mathbb{N}_{a + 1}$, and $t \in \mathbb{N}_{s}$. Denote by $$h_{\mu}(t, s) = \frac{H_{\mu}(t, \rho(s))}{H_{\mu}(t, a)}.$$ Then, the following hold:
\begin{enumerate}
\item[(I)] $0 < h_{\mu}(t, s)$. 
\item[(II)] If $\mu \geq 0$, then $h_{\mu}(t, s) \leq 1$, and if $-1 < \mu < 0$, then $h_{\mu}(t, s) > 1$. In particular, $h_{0}(t, s) = 1$.
\item[(III)] If $\mu > 0$, then $h_{\mu}(t, s)$ is an increasing function of $t$.
\item[(IV)] If $-1 < \mu < 0$, then $h_{\mu}(t, s)$ is a decreasing function of $t$.
\end{enumerate}
\end{proposition}

\begin{proof}
(I) First, consider 
\begin{equation}\label{h 1}
h_{\mu}(t, s) = \frac{(t - \rho(s))^{\overline{\mu}}}{(t - a)^{\overline{\mu}}} = \frac{\Gamma(t - s + \mu + 1)\Gamma(t - a)}{\Gamma(t - s + 1)\Gamma(t - a + \mu)}.
\end{equation}
Since $\Gamma(t - a)$, $\Gamma(t - a + \mu)$, $\Gamma(t - s + 1)$, $\Gamma(t - s + \mu + 1) > 0$, it follows from \eqref{h 1} that $h_{\mu}(t, s) > 0$.

(II) The proof of (II) follows from the monotonicity of $H_{\mu}(t, \rho(s))$ with respect to $s$. 

(III) Next, consider
\begin{align}
\nonumber \nabla h_{\mu}(t, s) & = \nabla \Big{[}\frac{(t - \rho(s))^{\overline{\mu}}}{(t - a)^{\overline{\mu}}}\Big{]} \\ \nonumber & = \frac{(t - s + 1)^{\overline{\mu}}}{(t - a)^{\overline{\mu}}} - \frac{(t - s)^{\overline{\mu}}}{(t - a - 1)^{\overline{\mu}}} \\ \nonumber & = \frac{\Gamma(t - s + \mu + 1)\Gamma(t - a)}{\Gamma(t - s + 1)\Gamma(t - a + \mu)} - \frac{\Gamma(t - s + \mu)\Gamma(t - a - 1)}{\Gamma(t - s)\Gamma(t - a + \mu - 1)} 
\end{align}
\begin{align} 
\nonumber & = \frac{\Gamma(t - s + \mu)\Gamma(t - a - 1)}{\Gamma(t - s)\Gamma(t - a + \mu - 1)} \Big{[}\frac{(t - s + \mu)(t - a - 1)}{(t - s)(t - a + \mu - 1)} - 1\Big{]} \\ & = \mu (s - a - 1)\frac{\Gamma(t - s + \mu)\Gamma(t - a - 1)}{\Gamma(t - s + 1)\Gamma(t - a + \mu)}. \label{h 2}
\end{align} 
Since $\Gamma(t - a - 1)$, $\Gamma(t - a + \mu)$, $\Gamma(t - s + \mu)$, $\Gamma(t - s + 1)$, $(s - a - 1) > 0$, it follows from \eqref{h 2} that $\nabla h_{\mu}(t, s) > 0$, implying that (III) holds.

(IV) Clearly, from \eqref{h 2}, we have 
\begin{equation}\label{h 3}
\nabla h_{-\mu}(t, s) = -\mu(s - a - 1)\frac{\Gamma(t - s - \mu)\Gamma(t - a - 1)}{\Gamma(t - s + 1)\Gamma(t - a - \mu)}.
\end{equation}
Since $\Gamma(t - a - 1)$, $\Gamma(t - a - \mu)$, $\Gamma(t - s - \mu)$, $\Gamma(t - s + 1) > 0$, $(s - a - 1) > 0$, it follows from \eqref{h 3} that $\nabla h_{-\mu}(t, s) < 0$, implying that (IV) holds.
\end{proof}

\section{Properties of Green's Function}

In this section, we obtain a few properties of $G(t, s)$ which we use in the later part of the article.

\begin{lemma}\label{Non Negative}
Assume $\alpha$, $\beta$, $\gamma$, $\delta \geq 0$ and $\beta \geq \alpha$ such that \eqref{Condition} holds.
\begin{enumerate}
\item $\xi > 0$ for all $t \in \mathbb{N}^{b}_{a}$.
\item $u(t, s) \geq 0$ for all $(t, s) \in \mathbb{N}^{b}_{a} \times \mathbb{N}^{b}_{a + 1}$ such that $t \leq s - 1$.
\item $v(t, s) \geq 0$ for all $(t, s) \in \mathbb{N}^{b}_{a} \times \mathbb{N}^{b}_{a + 1}$ such that $t \geq s$.
\end{enumerate}
\end{lemma}

\begin{proof}
(1) From Proposition \ref{H}, we have $H_{\nu - 1}(b, a)$, $H_{\nu - 2}(b, a) > 0$ implying that $$\xi = (\beta - \alpha)\gamma + \alpha \gamma H_{\nu - 1}(b, a) + \alpha \delta H_{\nu - 2}(b, a) > 0.$$
(2) From Proposition \ref{H}, we have $H_{\nu - 1}(b, \rho(s))$, $H_{\nu - 2}(b, \rho(s)) > 0$ for all $s \in \mathbb{N}^{b}_{a + 1}$, and $H_{\nu - 1}(t, a) \geq 0$ for all $t \in \mathbb{N}^{b}_{a}$. Also, from (1), we have $\xi > 0$ for all $t \in \mathbb{N}^{b}_{a}$. Thus, we obtain  
\begin{multline}
\nonumber u(t, s) = \frac{1}{\xi}\Big{[}\alpha \gamma H_{\nu - 1}(t, a)H_{\nu - 1}(b, \rho(s)) + \alpha \delta H_{\nu - 1}(t, a)H_{\nu - 2}(b, \rho(s)) \\  + (\beta - \alpha) \gamma H_{\nu - 1}(b, \rho(s)) + (\beta - \alpha) \delta H_{\nu - 2}(b, \rho(s))\Big{]} \geq 0,
\end{multline}
for all $(t, s) \in \mathbb{N}^{b}_{a} \times \mathbb{N}^{b}_{a + 1}$ such that $t \leq s - 1$.

(3) Consider
\begin{align}
\nonumber v(t, s) & = u(t, s) - H_{\nu - 1}(t, \rho(s)) \\ \nonumber & = \frac{1}{\xi}\Big{[}\alpha \gamma H_{\nu - 1}(t, a)H_{\nu - 1}(b, \rho(s)) + \alpha \delta H_{\nu - 1}(t, a)H_{\nu - 2}(b, \rho(s)) \\ & + (\beta - \alpha) \gamma H_{\nu - 1}(b, \rho(s)) + (\beta - \alpha) \delta H_{\nu - 2}(b, \rho(s)) - \xi H_{\nu - 1}(t, \rho(s))\Big{]} 
\end{align}
\begin{align} 
\nonumber & = \frac{1}{\xi}\Big{[}(\beta - \alpha) \delta H_{\nu - 2}(b, \rho(s)) + (\beta - \alpha)\gamma\Big{(}H_{\nu - 1}(b, \rho(s)) - H_{\nu - 1}(t, \rho(s))\Big{)} \\ \nonumber & + \alpha \delta \Big{(}H_{\nu - 1}(t, a)H_{\nu - 2}(b, \rho(s)) - H_{\nu - 1}(t, \rho(s))H_{\nu - 2}(b, a)\Big{)} \\ & + \alpha \gamma \Big{(}H_{\nu - 1}(t, a)H_{\nu - 1}(b, \rho(s)) - H_{\nu - 1}(b, a)H_{\nu - 1}(t, \rho(s))\Big{)} \Big{]} \label{Imp} \\ \nonumber & = \frac{1}{\xi}\Big{[}E_{1} + E_{2} + E_{3} + E_{4}\Big{]}.  
\end{align}
We already know that $\xi > 0$ for all $t \in \mathbb{N}^{b}_{a}$. Now, we show that $$E_{i} \geq 0, \quad i = 1, 2, 3, 4.$$ From Proposition \ref{H}, we have $H_{\nu - 2}(b, \rho(s)) > 0$ for all $s \in \mathbb{N}^{b}_{a + 1}$. So, we obtain $$E_{1} \geq 0.$$ Again, from Proposition \ref{H}, we have $H_{\nu - 1}(t, \rho(s)) \leq H_{\nu - 1}(b, \rho(s))$ for all $(t, s) \in \mathbb{N}^{b}_{a} \times \mathbb{N}^{b}_{a + 1}$ such that $t \geq s$, implying that $$E_{2} \geq 0.$$ From Proposition \ref{H}, we have $H_{\nu - 1}(t, \rho(s)) \leq H_{\nu - 1}(t, a)$, $H_{\nu - 2}(b, a) \leq H_{\nu - 2}(b, \rho(s))$ for all $(t, s) \in \mathbb{N}^{b}_{a} \times \mathbb{N}^{b}_{a + 1}$ such that $t \geq s$, implying that $$E_{3} \geq 0.$$ Now, consider
\begin{align*}
& H_{\nu - 1}(t, a)H_{\nu - 1}(b, \rho(s)) - H_{\nu - 1}(b, a)H_{\nu - 1}(t, \rho(s)) \\ & =  H_{\nu - 1}(b, a)H_{\nu - 1}(t, \rho(s))\Big{[}\frac{H_{\nu - 1}(b, \rho(s))}{H_{\nu - 1}(b, a)} \cdot \frac{H_{\nu - 1}(t, a)}{H_{\nu - 1}(t, \rho(s))} - 1\Big{]} \\ & =  H_{\nu - 1}(b, a)H_{\nu - 1}(t, \rho(s))\Big{[}\frac{h_{\nu - 1}(b, s)}{h_{\nu - 1}(t, s)} - 1\Big{]}.
\end{align*}
From Proposition \ref{H}, we have $H_{\nu - 1}(b, a)$, $H_{\nu - 1}(t, \rho(s)) > 0$, and $h_{\nu - 1}(b, s) > h_{\nu - 1}(t, s)$ for all $(t, s) \in \mathbb{N}^{b}_{a} \times \mathbb{N}^{b}_{a + 1}$ such that $t \geq s$, implying that $$E_{4} \geq 0.$$ Therefore, we obtain $v(t, s) \geq 0$ for all $(t, s) \in \mathbb{N}^{b}_{a} \times \mathbb{N}^{b}_{a + 1}$ such that $t \geq s$.
\end{proof}

\begin{theorem}\label{Non Negative G}
$G(t, s) \geq 0$ for all $(t, s) \in \mathbb{N}^{b}_{a} \times \mathbb{N}^{b}_{a + 1}$.
\end{theorem}

\begin{proof}
The proof follows from the preceding Lemma.
\end{proof}

\begin{lemma}\label{Monotonic}
Assume $\alpha$, $\beta$, $\gamma$, $\delta \geq 0$ and $\beta \geq \alpha$ such that \eqref{Condition} holds.
\begin{enumerate}
\item $u(t, s)$ is an increasing function of $t$ for all $(t, s) \in \mathbb{N}^{b}_{a} \times \mathbb{N}^{b}_{a + 1}$ such that $t \leq s - 1$.
\item $v(t, s)$ is a decreasing function of $t$ for all $(t, s) \in \mathbb{N}^{b}_{a} \times \mathbb{N}^{b}_{a + 1}$ such that $t \geq s$.
\end{enumerate}
\end{lemma}

\begin{proof}
(1) Consider $$\nabla_{t} u(t, s) = \frac{1}{\xi}\Big{[}\alpha \gamma H_{\nu - 2}(t, a)H_{\nu - 1}(b, \rho(s)) + \alpha \delta H_{\nu - 2}(t, a)H_{\nu - 2}(b, \rho(s))\Big{]}.$$ From Proposition \ref{H}, we have $H_{\nu - 1}(b, \rho(s))$, $H_{\nu - 2}(b, \rho(s)) > 0$ for all $s \in \mathbb{N}^{b}_{a + 1}$, and $H_{\nu - 2}(t, a) > 0$ for all $t \in \mathbb{N}^{b}_{a + 1}$. Also, from (1), we have $\xi > 0$ for all $t \in \mathbb{N}^{b}_{a + 1}$. Thus, we obtain $\nabla_{t} u(t, s) > 0$, implying that (1) holds.

(2) From \eqref{Imp}, we obtain 
\begin{align*}
\nabla_{t} v(t, s) & = \frac{1}{\xi}\Big{[}-(\beta - \alpha)\gamma H_{\nu - 2}(t, \rho(s)) \\ \nonumber & + \alpha \delta \Big{(}H_{\nu - 2}(t, a)H_{\nu - 2}(b, \rho(s)) - H_{\nu - 2}(t, \rho(s))H_{\nu - 2}(b, a)\Big{)} \\ & + \alpha \gamma \Big{(}H_{\nu - 2}(t, a)H_{\nu - 1}(b, \rho(s)) - H_{\nu - 1}(b, a)H_{\nu - 2}(t, \rho(s))\Big{)} \Big{]} \\ & = \frac{1}{\xi}\Big{[}E_{5} + E_{6} + E_{7}\Big{]}.
\end{align*}
Clearly, $\xi > 0$ for all $t \in \mathbb{N}^{b}_{a + 1}$. Now, we show that $$E_{i} \leq 0, \quad i = 5, 6, 7.$$ From Proposition \ref{H}, we have $H_{\nu - 2}(t, \rho(s)) > 0$ for all $(t, s) \in \mathbb{N}^{b}_{a} \times \mathbb{N}^{b}_{a + 1}$ such that $t \geq s$, implying that $$E_{5} \leq 0.$$ From Proposition \ref{H}, we have $H_{\nu - 2}(t, \rho(s)) \geq H_{\nu - 2}(t, a)$, $H_{\nu - 1}(b, a) \geq H_{\nu - 1}(b, \rho(s))$ for all $(t, s) \in \mathbb{N}^{b}_{a} \times \mathbb{N}^{b}_{a + 1}$ such that $t \geq s$, implying that $$E_{7} \leq 0.$$ Now, consider
\begin{align*}
& H_{\nu - 2}(t, a)H_{\nu - 2}(b, \rho(s)) - H_{\nu - 2}(t, \rho(s))H_{\nu - 2}(b, a) \\ & =  H_{\nu - 2}(t, \rho(s))H_{\nu - 2}(b, a)\Big{[}\frac{H_{\nu - 2}(b, \rho(s))}{H_{\nu - 2}(b, a)} \cdot \frac{H_{\nu - 2}(t, a)}{H_{\nu - 2}(t, \rho(s))} - 1\Big{]} \\ & =  H_{\nu - 2}(t, \rho(s))H_{\nu - 2}(b, a)\Big{[}\frac{h_{\nu - 2}(b, s)}{h_{\nu - 2}(t, s)} - 1\Big{]}.
\end{align*}
From Proposition \ref{H}, we have $H_{\nu - 2}(b, a)$, $H_{\nu - 2}(t, \rho(s)) > 0$, and $h_{\nu - 2}(t, s) > h_{\nu - 2}(b, s)$ for all $(t, s) \in \mathbb{N}^{b}_{a} \times \mathbb{N}^{b}_{a + 1}$ such that $t \geq s$, implying that $$E_{6} \leq 0.$$ Therefore, (2) holds.
\end{proof}

\begin{theorem}\label{G Bound}
Assume $\alpha$, $\beta$, $\gamma$, $\delta \geq 0$ and $\beta \geq \alpha$ such that \eqref{Condition} holds. The following inequality holds for the Green's function $G(t, s)$:
\begin{equation}
\max_{(t, s) \in \mathbb{N}^{b}_{a} \times \mathbb{N}^{b}_{a + 1}}G(t, s) < \Omega,
\end{equation}
for all $(t, s) \in \mathbb{N}^{b}_{a} \times \mathbb{N}^{b}_{a + 1}$, where
\begin{equation}
\Omega = \frac{1}{\xi}\Big{[}\alpha \gamma H_{\nu - 1}(b, a)H_{\nu - 1}(b, a) + \alpha \delta H_{\nu - 1}(b, a) \\ + (\beta - \alpha) \gamma H_{\nu - 1}(b, a) + (\beta - \alpha) \delta \Big{]}.
\end{equation}
\end{theorem}

\begin{proof}
Consider
\begin{multline}
\nonumber u(\rho(s), s) = \frac{1}{\xi}\Big{[}\alpha \gamma H_{\nu - 1}(\rho(s), a)H_{\nu - 1}(b, \rho(s)) + \alpha \delta H_{\nu - 1}(\rho(s), a)H_{\nu - 2}(b, \rho(s)) \\ + (\beta - \alpha) \gamma H_{\nu - 1}(b, \rho(s)) + (\beta - \alpha) \delta H_{\nu - 2}(b, \rho(s))\Big{]}, \quad s \in \mathbb{N}^{b}_{a + 1}.
\end{multline}
Denote by 
\begin{multline}
\nonumber f(s) = \frac{1}{\xi}\Big{[}\alpha \gamma H_{\nu - 1}(s, a)H_{\nu - 1}(b, \rho(s)) + \alpha \delta H_{\nu - 1}(s, a)H_{\nu - 2}(b, \rho(s)) \\ + (\beta - \alpha) \gamma H_{\nu - 1}(b, \rho(s)) + (\beta - \alpha) \delta H_{\nu - 2}(b, \rho(s))\Big{]}, \quad s \in \mathbb{N}^{b}_{a + 1}.
\end{multline} 
Then, by Lemma \ref{Non Negative} and Proposition \ref{H}, we have 
\begin{equation}\label{Zero}
0 \leq u(\rho(s), s) < f(s), \quad s \in \mathbb{N}^{b}_{a + 1}.
\end{equation}
Now, consider
\begin{align}
\nonumber v(s, s) & = \frac{1}{\xi}\Big{[}\alpha \gamma H_{\nu - 1}(s, a)H_{\nu - 1}(b, \rho(s)) + \alpha \delta H_{\nu - 1}(s, a)H_{\nu - 2}(b, \rho(s)) \\ & + (\beta - \alpha) \gamma H_{\nu - 1}(b, \rho(s)) + (\beta - \alpha) \delta H_{\nu - 2}(b, \rho(s))\Big{]} - 1  = f(s) - 1, \label{One}
\end{align}
for $s \in \mathbb{N}^{b}_{a + 1}$. By Lemma \ref{Non Negative}, we have $0 \leq v(s, s)$ for $s \in \mathbb{N}^{b}_{a + 1}$, implying that $1 \leq f(s)$ for $s \in \mathbb{N}^{b}_{a + 1}$. Then, from \eqref{One}, we obtain $$\frac{v(s, s)}{f(s)} = 1 - \frac{1}{f(s)} < 1, \quad s \in \mathbb{N}^{b}_{a + 1},$$ implying that
\begin{equation}\label{Two}
0 \leq v(s, s) < f(s), \quad s \in \mathbb{N}^{b}_{a + 1}.
\end{equation}
Since 
\begin{align*}
& \max_{s \in \mathbb{N}^{b}_{a + 1}}H_{\nu - 1}(s, a) = H_{\nu - 1}(b, a), \\ & \max_{s \in \mathbb{N}^{b}_{a + 1}}H_{\nu - 1}(b, \rho(s)) = H_{\nu - 1}(b, a), \quad \max_{s \in \mathbb{N}^{b}_{a + 1}}H_{\nu - 2}(b, \rho(s)) = 1, 
\end{align*}
we have 
\begin{equation}\label{Three}
f(s) < \Omega, \quad s \in \mathbb{N}^{b}_{a + 1}.
\end{equation}
Thus, by Proposition \ref{H}, \eqref{Zero}, \eqref{Two} and \eqref{Three}, we get
\begin{align*}
\max_{(t, s) \in \mathbb{N}^{b}_{a} \times \mathbb{N}^{b}_{a + 1}}G(t, s) & = \max_{s \in \mathbb{N}^{b}_{a + 1}}\big{\{}u(\rho(s), s), v(s, s)\big{\}} \\ & \leq \Big{\{}\max_{s \in \mathbb{N}^{b}_{a + 1}}u(\rho(s), s), \max_{s \in \mathbb{N}^{b}_{a + 1}}v(s, s)\Big{\}} < \max_{s \in \mathbb{N}^{b}_{a + 1}}f(s) < \Omega.
\end{align*}
\end{proof}

\begin{theorem}\label{Sum G Bound}
Assume $\alpha$, $\beta$, $\gamma$, $\delta \geq 0$ and $\beta \geq \alpha$ such that \eqref{Condition} holds. The following inequality holds for the Green's function $G(t, s)$:
\begin{equation}\label{GB}
\sum^{b}_{s = a + 1}G(t, s) < \Lambda,
\end{equation}
for all $(t, s) \in \mathbb{N}^{b}_{a} \times \mathbb{N}^{b}_{a + 1}$, where
\begin{multline}
\Lambda = \frac{1}{\xi}\Big{[} \alpha \gamma H_{\nu - 1}(b, a) H_{\nu}(b, a) + \alpha \delta H_{\nu - 1}(b, a)H_{\nu - 1}(b, a) \\ + (\beta - \alpha) \gamma H_{\nu}(b, a) + (\beta - \alpha)\delta H_{\nu - 1}(b, a)\Big{]}.
\end{multline}
\end{theorem}

\begin{proof}
Consider
\begin{align*}
\sum^{b}_{s = a + 1}G(t, s) & = \sum^{t}_{s = a + 1}v(t, s) + \sum^{b}_{s = t + 1}u(t, s) \\ & = \sum^{b}_{s = a + 1}u(t, s) - \sum^{t}_{s = a + 1}H_{\nu - 1}(t, \rho(s)) \\ & = \frac{1}{\xi}\Big{[}\alpha \gamma H_{\nu - 1}(t, a) \sum^{b}_{s = a + 1}H_{\nu - 1}(b, \rho(s)) \\ & + \alpha \delta H_{\nu - 1}(t, a) \sum^{b}_{s = a + 1}H_{\nu - 2}(b, \rho(s)) + (\beta - \alpha) \gamma  \sum^{b}_{s = a + 1}H_{\nu - 1}(b, \rho(s)) \\ & + (\beta - \alpha) \delta  \sum^{b}_{s = a + 1}H_{\nu - 2}(b, \rho(s))\Big{]} - \sum^{t}_{s = a + 1}H_{\nu - 1}(b, \rho(s)) \\ & = \frac{1}{\xi}\Big{[}\alpha \gamma H_{\nu - 1}(t, a) H_{\nu}(b, a) + \alpha \delta H_{\nu - 1}(t, a)H_{\nu - 1}(b, a) \\ & + (\beta - \alpha) \gamma  H_{\nu}(b, a) + (\beta - \alpha) \delta  H_{\nu - 1}(b, a)\Big{]} - H_{\nu}(t, a) \\ & = \frac{1}{\xi}\Big{[} \alpha \gamma \Big{(}H_{\nu - 1}(t, a) H_{\nu}(b, a) - H_{\nu - 1}(b, a)H_{\nu}(t, a) \Big{)} \\ & + \alpha \delta \Big{(}H_{\nu - 1}(t, a)H_{\nu - 1}(b, a) - H_{\nu}(t, a)H_{\nu - 2}(b, a)\Big{)} \\ & + (\beta - \alpha) \gamma \Big{(}H_{\nu}(b, a) - H_{\nu}(t, a)\Big{)} + (\beta - \alpha) \delta H_{\nu - 1}(b, a)\Big{]}.
\end{align*}
Since $H_{\nu}(t, a) \geq 0$ for all $t \in \mathbb{N}^{b}_{a}$ and $$\max_{t \in \mathbb{N}^{b}_{a}}H_{\nu}(t, a) = H_{\nu}(b, a), \quad \max_{t \in \mathbb{N}^{b}_{a}}H_{\nu - 1}(t, a) = H_{\nu - 1}(b, a),$$ we obtain \eqref{GB}. The proof is complete.
\end{proof}

\begin{theorem}[See \cite{Br}]\label{Unique}
Let $h : \mathbb{N}^{b}_{a + 1} \rightarrow \mathbb{R}$. If \eqref{Gen Hom} has only the trivial solution, then the nonhomogeneous boundary value problem
\begin{equation}\label{Gen Non Hom 1}
\begin{cases}
-\big{(}\nabla^{\alpha}_{a}u\big{)}(t) = h(t), \quad t \in \mathbb{N}^{b}_{a + 2},\\
\alpha u(a + 1) - \beta (\nabla u)(a + 1) = 0,\\
\gamma u(b) + \delta (\nabla u)(b) = 0,
\end{cases}
\end{equation} 
has a unique solution given by
\begin{equation}
u(t) = \sum^{b}_{s = a + 1}G(t, s)h(s), \quad t \in \mathbb{N}^{b}_{a}.
\end{equation}
\end{theorem}

Now, we are able to establish a Lyapunov-type inequality for the nabla fractional boundary value problem \eqref{Gen Non Hom}.

\begin{theorem}
Assume $\alpha$, $\beta$, $\gamma$, $\delta \geq 0$ and $\beta \geq \alpha$ such that \eqref{Condition} holds. If the nabla fractional boundary value problem \eqref{Gen Non Hom} has a nontrivial solution, then
\begin{equation}
\sum^{b}_{s = a + 1}|q(s)| > \frac{1}{\Omega}.
\end{equation}
\end{theorem}

\begin{proof}
Let $\mathcal{B}$ be the Banach space of functions endowed with norm $$\|u\| : = \max_{t \in \mathbb{N}^{b}_{a}}|u(t)|.$$ It follows from the above Theorem that a solution to \eqref{Gen Non Hom} satisfies the equation $$u(t) = \sum^{b}_{s = a + 1}G(t, s)q(s)u(s), \quad t \in \mathbb{N}^{b}_{a}.$$ Hence
\begin{align*}
\|u\| = \max_{t \in \mathbb{N}^{b}_{a}}|u(t)| & = \max_{t \in \mathbb{N}^{b}_{a}}\Big{|}\sum^{b}_{s = a + 1}G(t, s)q(s)u(s)\Big{|} \\ & \leq \max_{t \in \mathbb{N}^{b}_{a}}\Big{[}\sum^{b}_{s = a + 1}G(t, s)|q(s)||u(s)|\Big{]} \\ & \leq \|u\| \max_{t \in \mathbb{N}^{b}_{a}}\Big{[}\sum^{b}_{s = a + 1}G(t, s)|q(s)|\Big{]} \\ & < \Omega \|u\| \sum^{b}_{s = a + 1}|q(s)|, \quad (\text{using Theorem \ref{G Bound}})
\end{align*}
or, equivalently, $$\sum^{b}_{s = a + 1}|q(s)| > \frac{1}{\Omega}.$$ The proof is complete.
\end{proof}

\end{document}